\documentclass[a4paper,12pt]{article}
\usepackage[left=2cm,right=2cm,top=2cm,bottom=2cm]{geometry}
\usepackage{comment, todonotes,colortbl,xcolor}

\usepackage{amsmath, amssymb, amsthm}

\DeclareMathOperator{\tr}{\mathrm{tr}\,}

\newcommand{\R}{\mathbb{R}}

\newcommand{\cof}{\mathrm{cof}\,}

\newtheorem{theorem}{Theorem}

\newtheorem{lemma}{Lemma}
\theoremstyle{definition}

\newtheorem{remark}{Remark}
\newtheorem{corollary}{Corollary}
\title{Polyconvex double well functions}

\begin{document}
 
\author{Didier Henrion$^{1,2}$, Martin Kru\v{z}\'{i}k$^{3,4}$}
\date{\today}

\footnotetext[1]{LAAS-CNRS, Universit\'e de Toulouse, France} 
\footnotetext[2]{Faculty of Electrical Engineering, Czech Technical University in Prague, Czechia}
\footnotetext[3]{Czech Academy of Sciences, Institute of Information Theory and Automation,  Prague, Czechia}
\footnotetext[4]{Faculty of Civil Engineering, Czech Technical University in Prague, Czechia}

	\maketitle
\begin{abstract}
	We investigate polyconvexity of the double well function $f(X) := |X-X_1|^2|X-X_2|^2$ for given matrices $X_1, X_2 \in \R^{n \times n}$. Such functions are fundamental in the modeling of phase transitions in materials, but their non-convex nature presents challenges for the analysis of variational problems.  Polyconvexity of $f$ is related to the singular values of the matrix difference $X_1 - X_2$. We prove that $f$ is polyconvex if and only if the square of the largest singular value does not exceed the sum of the squares of the other singular values.  This condition allows the function to be decomposed into the sum of a strictly convex part and a null Lagrangean. As a direct application of this result, we prove an existence and uniqueness theorem for the corresponding Dirichlet minimization problem of the integral functional.
\end{abstract}
	
	\section{Introduction}

A central challenge in the calculus of variations is to establish the existence of solutions to variational problems, typically formulated as minimizing an integral functional of the form
\[
I(y) := \int_\Omega f(\nabla y(x)) \, dx,
\]
where $\Omega \subset \R^n$ is a  bounded Lipschitz domain, $y: \Omega \to \R^m$ is a deformation, and $f: \R^{m \times n} \to \R$ is a  continuous stored energy function.
For scalar-valued functions and/or one-dimensional domains  ($m=1$ and/or $n=1$), the existence of minimizers can often be guaranteed under the assumption that the integrand $f$ is convex. However, in many physical applications, particularly in nonlinear elasticity and materials science,  $y$ is vector-valued ($m=n>1$)  and $f$ is non-convex. This non-convexity comes from physical grounds and is essential for modeling complex material behaviors such as phase transitions, where a material can exist in multiple stable states, represented by different energy wells,  see e.g. \cite{sirev,m3as}. 

A groundbreaking contribution to address this challenge was made by John Ball in 1977 \cite{Ball1977}. He  introduced the weaker notion of polyconvexity. This condition is strong enough to ensure the weak lower semicontinuity of the energy functional $I$;  a key requirement for the direct method in the calculus of variations -- yet flexible enough to include physically relevant non-convex energy functions. Formally, a function $f: \R^{n \times n} \to \R$ is polyconvex if it can expressed as a convex function of all the minors (i.e. sub-determinants of all orders) of its matrix argument, i.e., $f(X)=h(T(X))$, where $T:\R^{n\times n}\to \R^{{2n\choose n} -1}$ is the vector map  returning all subdeterminants of the argument  and $h:\R^{{2n\choose n} -1}\to\R$ is convex.   Polyconvexity implies (Morrey's) quasiconvexity, which is a weaker condition and which is equivalent to sequential weak  lower semicontinuity of $I$ on $W^{1, p}(\Omega;\R^n)$ whenever $0\le f\le C(1+|\cdot|^p)$ for some $C>0$ is continuous.  We say that such $f$ is quasiconvex  if 
\begin{align}\label{qc}
    \int_\Omega f(\nabla\varphi(x))\, dx\ge \mathcal{L}^n(\Omega)f(A) 
\end{align}
for all $A\in\R^{n\times n}$ and all $\varphi\in W^{1, p}(\Omega;\R^n)$ such that $\varphi(x)=Ax$ for $x\in\partial\Omega$. It was shown in \cite[Prop.~2.4]{ball-murat} that  in this case one can equivalently consider $\varphi$ from the smaller space $W^{1,\infty}(\Omega;\R^n)$.   
The theory and its applications are now standard material in advanced monographs on the subject, such as \cite{Ciarlet1988}, \cite{Dacorogna2008},  or \cite{kruzik-roubicek}. 

In this paper, we investigate the polyconvexity of a specific and widely used {double well energy function}:
\[
f(X) := |X - X_1|^2 |X - X_2|^2,
\]
where $X_1$ and $X_2$ are two given matrices in $\R^{n \times n}$ representing the preferred states or phases of a material.

The paper is structured as follows. We begin in Sections \ref{sec:dim2} and \ref{sec:dim3} by analyzing the polyconvexity of this function for the specific cases of $n=2$ and $n=3$ with  $X_1=-X_2 = I_n$, the identity matrix in $n$ dimensions,  demonstrating the core mechanisms through direct computations. In Section \ref{sec:poly}, we present our main result, which provides a necessary and sufficient condition for the polyconvexity of the double well function in any dimension $n>1$.

Polyconvexity of $f$ is related to the singular values of the matrix difference $X_1 - X_2$. We prove that $f$ is polyconvex if and only if the square of the largest singular value does not exceed the sum of the squares of the other singular values.

The proof reveals that under this condition, the function can be decomposed into the sum of a strictly convex function and a null Lagrangean \cite{Dacorogna2008}.
Finally, as a direct and significant application of this result, in Section \ref{sec:cov} we prove existence and uniqueness of a solution to the corresponding Dirichlet problem in the calculus of variations. This demonstrates that, despite its non-convexity and the double-well structure, the energy functional is well-behaved under the identified  spectral  condition \eqref{spectral-condition}, guaranteeing that the associated variational problem is well-posed and allows for a unique minimizer.

We use the notation  $\tr X$ for the trace of a matrix $X\in\R^{n\times n}$, $|X|^2 = \tr X^T X$ for the squared Frobenius norm, and $I_n$ for the identity matrix in $\R^{n\times n}$. Moreover,
$\cof X\in \R^{n\times n}$ is the cofactor matrix of $X$, i.e., $(\cof X)_{ij}=(-1)^{i+j}\det X'_{ij}$ where $X'_{ij}$ is the submatrix of $X$ obtained by removing the $i$-th row and the $j$-th column from $X$.  Further we denote $\mathrm{SO}(n)=\{R\in\R^{n\times n}:\ R^T R=RR^T =I_n,\ \det R=1\}$ the set of  rotations, and  $\mathrm{O}(n)= \{R\in\R^{n\times n}:\ R^T R=RR^T =I_n,\ \det R=\pm 1\}$ the set of all orthonormal matrices. The dyadic product of $u,v\in\R^n$ is denoted $u\otimes v\in\R^{n\times n}$ with $(u\otimes v)_{ij}=u_iv_j$ for all $i,j\in\{1,\ldots, n\}$.  We denote $\mathcal{L}^m$ the $m$-dimensional Lebesgue measure and $W^{1,p}(\Omega;\R^n)$ for $1\le p\le+\infty$, refers to   Sobolev spaces of functions defined on $\Omega$ with values in $\R^n$.
	
\section{The 2x2 case}\label{sec:dim2}
	
\begin{lemma}
The double well function $X \in \R^{2\times 2} \mapsto f(X) := |X-I_2|^2|X+I_2|^2$ is polyconvex.
\end{lemma}
 
	\begin{proof}
First expand the function:
	\begin{align*}
		f(X) &= |X-I_2|^2|X+I_2|^2 \\
		&= \left( |X|^2 - 2\tr X +  2 \right) \left( |X|^2 + 2\tr X + 2 \right) \\
		&= \left( (|X|^2 + 2) - 2\tr X \right) \left( (|X|^2 + 2) + 2\tr X \right) \\
		&= (|X|^2 + 2)^2 - 4(\tr X)^2 \\
		&= |X|^4 + 4|X|^2 + 4 - 4(\tr X)^2.
	\end{align*}
Now observe that
\[ (\tr X)^2 = |X|^2 + 2\det X - (X_{21}-X_{12})^2. \] Substituting this identity into the expression for $f(X)$ yields: \begin{align*} f(X) &= |X|^4 + 4|X|^2 + 4 - 4\left( |X|^2 + 2\det X - (X_{21}-X_{12})^2 \right) \\ &= |X|^4 + 4|X|^2 + 4 - 4|X|^2 - 8\det X + 4(X_{21}-X_{12})^2 \\ &= |X|^4 + 4(X_{21}-X_{12})^2 - 8\det X + 4\end{align*}
which is the sum of a strictly convex  function of $X$ and a linear function of $\det X$.
		\end{proof}

\subsection{A frame-invariant version}

A function $g : \R^{n\times n} \to \R$ is frame invariant if $g(RX)=g(X)$ for all rotation matrices $R$, i.e. all orthonormal matrices with determinant  $+1$, and all $X\in \R^{n\times n}$

\begin{lemma}
The function  $g:\R^{2\times 2}\to\R$ defined as  $g(X)=f(X)-4(X_{21}-X_{12})^2$ is polyconvex, frame invariant, non-negative, and vanishes exactly at  rotations, i.e., also at  $X=\pm I_2$.
\end{lemma}

\begin{proof}
Polyconvexity and frame invariance follow from the expression
\[
g(X) := |X|^4+4-8\det X.
\]
Non-negativity of $g$ follows from the polynomial sums of squares (SOS) representation
\[
g(X) = (|X|^2-2)^2 + 4(X_{11}-X_{22})^2 + 4(X_{21}  +  X_{12})^2
\]
which also implies $g(I_2) = g(-I_2)=0$.
\end{proof}

\begin{remark}
Note that $g$ may be seen as a version of the Saint Venant-Kirchhoff stored energy density $W(X)=\frac12|X^T X-I_2|^2$, see, e.g. ,\cite{Dacorogna2008}. However, $W(Q)=0$ while $g(Q)>0$ for $Q\in \mathrm{O}(2)\setminus\mathrm{SO}(2)$, so that $g$ is not minimized on mechanically unacceptable deformations that change the orientation (dark wells).
\end{remark}

	\section{The 3x3 case}\label{sec:dim3}
		
	\begin{lemma}
		The double well function $X \in \R^{3\times 3} \mapsto f(X) := |X-I_3|^2|X+I_3|^2$ is polyconvex.
	\end{lemma}
	
	\begin{proof}
		First expand the function:
			\begin{align*}
		f(X) &= |X|^4 + 6|X|^2 + 9 - 4(\tr X)^2 \\
		& =  (|X|^2+3)^2-(2\tr X)^2.
		\end{align*}

From the identity
	\[
	\det(\lambda I_3 + X) = \lambda^3 + \lambda^2 \tr X  + \lambda \tr\cof X + \det X
	\]
	we define $\tr\cof X$ as  the trace of the  cofactor  matrix of $X$ or equivalently  the sum of the principal second-order minors of $X$.
	We substitute the identity
	\[
	(\tr X)^2 = \tr X^2 + 2 \tr\cof X
	\]
	into the expression for $f(X)$ to get:
	\begin{align*}
		f(X) &= |X|^4 + 6|X|^2 + 9 - 4(\tr X^2 + 2 \tr\cof X) \\
		&= |X|^4 + 6|X|^2 - 4\tr X^2 - 8 \tr\cof X + 9.
	\end{align*}
Decomposing $X$ into its symmetric part $X_s = \frac{1}{2}(X+X^T)$ and its skew-symmetric part $X_a = \frac{1}{2}(X-X^T)$, it holds $|X|^2 = |X_s|^2 + |X_a|^2$ and $\tr X^2 = |X_s|^2 - |X_a|^2$. This implies that the quadratic form
$|X|^2 - \tr X^2 = 2|X_{ a}|^2$ is non-negative and hence convex. Then
\[
f(X) = |X|^4 + 2|X|^2 + 8|X_{ a}|^2 - 8 \tr\cof X + 9
\]
is the sum of a strictly convex function of $X$ and a linear function of the second-order minors of $X$.

        \end{proof}

        \begin{corollary} [Isotropy]

        It holds for any rotation matrix $R\in \mathrm{SO}(3)$ and any $X\in\R^{3\times 3}$
        $$
        f(X)=f(RXR^T) 
        $$
      \end{corollary}  

      \begin{proof}
       It follows from the identity  $$\cof(RXR^T)=\cof R\,\cof X\,\cof R^T= R\,\cof X \,R^T  $$
       and from the fact that similar matrices $\cof(RXR^T)$ and $\cof X$  have the same   trace. 
      \end{proof}
	

\subsection{Frame invariant version}

\begin{lemma}
	The function $g:\R^{3\times 3}\to\R$, $g(X)= f(X) + 4 (\tr X)^2 -9$ is convex, frame invariant, non-negative and vanishing at $X=0$.
	The function $\hat g:\R^{3\times 3}\to\R$,  $\hat g(X)= f(X) + 4 (\tr X)^2 - 12 |X|^2$ is frame invariant, non-negative,  vanishing at $X=\pm I_3$, but it is not polyconvex.
\end{lemma}

\begin{proof}
The statements of the first sentence follow  from the expression	
\[
	f(X) + 4(\tr X)^2 - 9 = |X|^4 + 6|X|^2 = |X|^2(|X|^2+6).
	\]
The statements of the second sentence
follow from the SOS decomposition
\[
	f(X) + 4(\tr X)^2 - 12 |X|^2  
	= (|X|^2-3)^2,
\]
hence $9=\hat g(0)>\frac12(\hat g(\mathrm {diag }(\sqrt{3}, 0, 0))+\hat g(\mathrm {diag }(-\sqrt{3}, 0, 0)))=0$. Therefore, $\hat g$ is not rank-one convex (see \cite{Dacorogna2008})  and, consequently, not polyconvex.
\end{proof}

\section{Double well polyconvexity}\label{sec:poly}

More generally, given two matrices $X_1$ and $X_2$ of $\R^{n\times n}$ with $n \geq 2$, consider the double well function
\begin{align}\label{function}
X \in \R^{n\times n} \mapsto
f(X):=|X-X_{1}|^{2}\,|X-X_{2}|^{2}.
\end{align}

\begin{theorem}\label{polyconvex}
 
Let $\sigma_1 \geq \sigma_2 \geq \cdots \geq \sigma_n \geq 0$ denote the singular values of the matrix difference $X_1-X_2$.
Function $f$ is polyconvex if and only if the following spectral condition \begin{align}\label{spectral-condition}
\sigma_1^2+\sigma^2_2 + \cdots + \sigma^2_n \geq 2\sigma^2_1
\end{align}
holds.

\end{theorem}

\begin{proof}
Let
\[
A \;:=\;\tfrac12\,(X_{1}-X_{2}), 
\quad 
B \;:=\;\tfrac12\,(X_{1}+X_{2})
\]

so that
\begin{equation}\label{eq:fab}
f(X) = |(X-B)-A|^2|(X-B)+A|^2
\end{equation}
and consider the singular value decomposition
\[
A = U \Sigma V^T
\]
with $U, V \in \mathrm{O}(n)$, $\Sigma = \mathrm{diag}\:(\sigma_1, \sigma_2, \ldots, \sigma_n)$, $\sigma_1 \geq \sigma_2 \geq \cdots \geq \sigma_n \geq 0$.
Upon defining
\begin{align}\label{xz}
    Z:=U^T(X-B)V,
\end{align}
the terms in \eqref{eq:fab} can be rewritten as
\[
|(X-B) \pm A| = |U^T((X-B) \pm A)V| = |U^T(X-B)V \pm U^TAV| = |Z \pm \Sigma|.
\]
It follows from the Cauchy-Binet formula for the determinant of the product of two matrices (see e.g.~\cite[p. 103]{marcus1}) and from the formula for the determinant of the sum of two matrices (see e.g.~\cite[p.~162]{marcus2}) that  polyconvexity is invariant under left and right orthogonal transformations and under translation. Consequently,  polyconvexity of $f$ is equivalent to polyconvexity of the function $g:\R^{n\times n}\to\R$:
\begin{equation}\label{g}
f(X)=g(Z) := |Z-\Sigma|^{2}\,|Z+\Sigma|^{2},
\end{equation}
where $X$ and $Z$ are related by \eqref{xz}.

    Denoting $\langle A,B \rangle := \tr A^T B$, it holds
    \[
    |Z\pm\Sigma|^2 = \underbrace{|Z|^2+|\Sigma|^2}_{a}\pm2\underbrace{\langle Z,\Sigma\rangle}_{b}
    \]
and
\begin{equation}\label{gb}
g(Z) = (a-b)(a+b) = a^2-b^2 = (|Z|^2+|\Sigma|^2)^2 - 4\langle Z,\Sigma\rangle^2.
\end{equation}
Now expand
\[
(|Z|^2+|\Sigma|^2)^2 = |Z|^4 + 2|\Sigma|^2|Z|^2 + |\Sigma|^4
\]
and since $\Sigma$ is diagonal,
\[
4\langle Z,\Sigma\rangle^2 = 4(\sum_{1\leq i\leq n} \sigma_i Z_{ii})^2 = 
4\sum_{1\leq i\leq n} \sigma_i^2 Z_{ii}^2 + 8\sum_{1\leq i<j\leq n}\sigma_i\sigma_j Z_{ii}Z_{jj}.        
\]
Substracting yields 
\begin{equation}\label{eq:g2}
g(Z)=|Z|^4 + 2|\Sigma|^2|Z|^2 + |\Sigma|^4
- 4\sum_{i=1}^n \sigma_i^2 Z_{ii}^2 - 8\sum_{1\le i<j\le n}\sigma_i\sigma_j Z_{ii}Z_{jj}.
\end{equation}
Split the quadratic term into diagonal/off-diagonal parts:
\begin{equation}\label{eq:a}
2|\Sigma|^2|Z|^2
=2|\Sigma|^2\sum_{i=1}^n Z_{ii}^2 \;+\; 2|\Sigma|^2\sum_{i\ne j} Z_{ij}^2
=2|\Sigma|^2\sum_{i=1}^n Z_{ii}^2 \;+\; 2|\Sigma|^2\sum_{i<j}(Z_{ij}^2+Z_{ji}^2). 
\end{equation}
Collect the diagonal terms from \eqref{eq:g2} and \eqref{eq:a}:
\begin{equation}\label{eq:b}
2|\Sigma|^2\sum_{i=1}^n Z_{ii}^2 - 4\sum_{i=1}^n \sigma_i^2 Z_{ii}^2
=2\sum_{i=1}^n (|\Sigma|^2-2\sigma_i^2)\,Z_{ii}^2
\end{equation}
and pairwise complete the off-diagonal entries in \eqref{eq:a}, i.e., for each \(i<j\):
\begin{equation}\label{eq:c}
\begin{aligned}
2|\Sigma|^2(Z_{ij}^2+Z_{ji}^2)
&=2
\begin{bmatrix}Z_{ij}&Z_{ji}\end{bmatrix}
\!\begin{pmatrix}|\Sigma|^2&0\\[2pt]0&|\Sigma|^2\end{pmatrix}\!
\begin{bmatrix}Z_{ij}&Z_{ji}\end{bmatrix}^T\\
&=2
\begin{bmatrix}Z_{ij}&Z_{ji}\end{bmatrix}
\!\underbrace{\begin{pmatrix}|\Sigma|^2&-2\sigma_i\sigma_j\\[2pt]-2\sigma_i\sigma_j&|\Sigma|^2\end{pmatrix}}_{M_{ij}}\!
\begin{bmatrix}Z_{ij}& Z_{ji}\end{bmatrix}^T
\;+\;8\,\sigma_i\sigma_j\,Z_{ij}Z_{ji}.
\end{aligned}
\end{equation}
Substitute into \eqref{eq:g2} and regroup, using \eqref{eq:b} for the diagonal part and \eqref{eq:c} inside \eqref{eq:a} for the off–diagonal part:
\begin{equation}\label{eq:d}
\begin{aligned}
g(Z)
&=|Z|^4 + |\Sigma|^4
+ 2\sum_{i=1}^n (|\Sigma|^2-2\sigma_i^2)\,Z_{ii}^2 \\
&\quad
+ 2\sum_{1\le i<j\le n}
\begin{bmatrix}Z_{ij}&Z_{ji}\end{bmatrix}
M_{ij}
\begin{bmatrix}Z_{ij}&Z_{ji}\end{bmatrix}^T
\\
&\quad
+ 8\sum_{1\le i<j\le n}\sigma_i\sigma_j Z_{ij}Z_{ji}
- 8\sum_{1\le i<j\le n}\sigma_i\sigma_j Z_{ii}Z_{jj}.
\end{aligned}
\end{equation}
Introducing
\[
s_2(Z):=\sum_{1\le i<j\le n}\sigma_i\sigma_j\,(Z_{ii}Z_{jj}-Z_{ij}Z_{ji}) = \sum_{1\le i<j\le n}\sigma_i\sigma_j\,\det
\begin{pmatrix}Z_{ii}&Z_{ij} \\
Z_{ji} & Z_{jj}\end{pmatrix}
\]
we rewrite
\[
8\sum_{i<j}\sigma_i\sigma_j Z_{ij}Z_{ji}
- 8\sum_{i<j}\sigma_i\sigma_j Z_{ii}Z_{jj}
= -\,8\,s_2(Z)
\]
and substitute into \eqref{eq:d} to get
\begin{equation}\label{eq:gfull}
g(Z) = |Z|^4 + |\Sigma|^4 
+ 2\sum_{i=1}^n (|\Sigma|^2-2\sigma_i^2)\,Z_{ii}^2
+ 2\sum_{1\le i<j\le n}
\begin{bmatrix}Z_{ij}&Z_{ji}\end{bmatrix}
M_{ij}
\begin{bmatrix}Z_{ij}&Z_{ji}\end{bmatrix}^T
- 8\,s_2(Z).
\end{equation}
Now if $\sigma^2_2 + \cdots + \sigma^2_n \geq \sigma^2_1$ then $|\Sigma|^2 = \sigma^2_1 + \sigma^2_2 + \cdots + \sigma^2_n \geq 2 \sigma^2_1$, so the terms $(|\Sigma|^2-2\sigma_i^2)\,Z_{ii}^2$ in \eqref{eq:gfull} are positive and convex in $Z$. For each $(i,j)$, the eigenvalues of the matrix $M_{ij}$ defined in \eqref{eq:c} are $|\Sigma|^2\pm2\sigma_i\sigma_j$, and $|\Sigma|^2 \ge 2\sigma^2_1 \geq \sigma_i^2+\sigma_j^2\ \ge\ 2\sigma_i\sigma_j$. This implies that all matrices $M_{ij}$ are positive semidefinite, and hence the terms $\begin{bmatrix}Z_{ij}&Z_{ji}\end{bmatrix}
M_{ij}
\begin{bmatrix}Z_{ij}&Z_{ji}\end{bmatrix}^T$ in \eqref{eq:gfull} are positive and convex in $Z$. The term $|Z|^4$ is strictly convex in $Z$. The remaining term is $-8\,s_2(Z)$, which is linear in the principal $2\times2$ minors of $Z$.
By definition, this means that $g$ is polyconvex, and that proves the sufficiency of the spectral condition $\sigma^2_2 + \cdots + \sigma^2_n \geq \sigma^2_1$.

To prove necessity of the spectral condition, we will use the Legendre-Hadamard inequality. Polyconvexity of $g$ implies rank‑one convexity \cite[Chapter 5]{Dacorogna2008}, i.e., positivity of the second derivative of $g$ in the rank-one directions:
\[
\nabla^2 g(Z)[u\otimes v,u\otimes v] \ge0,
\quad\forall\,Z \in \R^{n\times n}, \forall u,v \in \R^{n}.
\]
Recalling equation \eqref{gb}, we have $g(Z)=(|Z|^2+|\Sigma|^2)^2-4\langle Z,\Sigma\rangle^2$ and at $Z=0$ we get the inequality
\begin{equation}\label{positivity}
\nabla^2 g(0)[u\otimes v,u\otimes v]=4(|\Sigma|^2\,|u\otimes v|^2-2\langle u \otimes v,\Sigma\rangle^2) \geq 0
\end{equation}
Now $|u\otimes v|^2=|u|^2|v|^2$, and since $\Sigma=\mathrm{diag}(\sigma_1,\dots,\sigma_n)$, it holds
\[
\langle u\otimes v,\Sigma\rangle=\sum_{i}\sigma_i\,u_i v_i
= \langle\Sigma\,u,v\rangle.
\]
Thus inequality \eqref{positivity} reads
\[
|\Sigma|^2 \,|u|^2|v|^2\ \ge\ 2\,\langle\Sigma\,u,v\rangle^2,
\qquad\forall u,v \in\R^n.
\]
For unit vectors  $u$ and $v$, the quantity on the right is maximized by aligning $u$ and $v$ with the top singular direction, i.e. $u=v=e_1$, which yields
\[
\sigma^2_1+\sigma^2_2+\cdots+\sigma^2_n \ge 2\,\sigma_1^2
\]
which is our spectral condition   \eqref{spectral-condition}.

\end{proof}

\begin{remark}\label{remark}
    If $n=2$  we get from \eqref{spectral-condition} that  $\sigma_1=\sigma_2$ and if $n=3$ we get the triangle inequality $\sigma_i^2+\sigma_j^2\ge\sigma_k^2$ for all  mutually distinct indices $i,j,k\in\{1,2,3\}$. 
\end{remark}

\begin{remark}

    Polyconvexity of $g$ from \eqref{gb} is guaranteed if the weaker condition of rank-one convexity holds at the most critical point, namely zero,  the local maximum of $g$ on the line segment $[-\Sigma,\Sigma]$.
\end{remark}

\section{Existence and uniqueness of minimizers}\label{sec:cov}

\begin{theorem} Let $\Omega \subset \R^n$, $n \ge 2$, be a bounded, open set with a Lipschitz boundary.  Consider a double well energy density function $f$ defined in \eqref{function} and satisfying the spectral condition \eqref{spectral-condition}.  Given a function $y_0 \in W^{1,4}(\Omega; \R^n)$ and the total energy functional 
$$I(y):= \int_\Omega f(\nabla y(x)) \, dx, $$ the minimization problem 
	\begin{align}\label{minI}
	  \min_{y \in W^{1,4}(\Omega; \R^n)} I(y) \text{ s.t. } y = y_0 \text{ on } \partial\Omega    
	\end{align} 
    admits a  solution and 
    the minimizer is unique.  \end{theorem}

\begin{proof}
As in the proof of Theorem \ref{polyconvex},
let $A:=\tfrac12(X_1-X_2)$, $B:=\tfrac12(X_1+X_2)$ and the singular value decomposition
$A=U\Sigma V^T$ with $\Sigma=\mathrm{diag}(\sigma_1,\dots,\sigma_n)$ and
$\sigma_1\ge\cdots\ge\sigma_n\ge0$. Let $y\in W^{1,4}(\Omega;\R^n)$ with $y=y_0$ on $\partial\Omega$.
We define 
\[
Z(x):=U^T(\nabla y(x)-B)\,V.
\]

Note that $Z=\nabla u\circ V^{-1}$, i.e., $Z(x)=\nabla u(V^{-1}x)$, where  $u\in W^{1,4}(V^{-1}\Omega;\R^n)$ and  $[u\circ V^{-1}]\in W^{1,4}(\Omega;\R^n)$
$$u(V^{-1}x):=U^\top y(x)-U^\top Bx.$$
Here $V^{-1}\Omega:=\{V^{-1}x:\, x\in\Omega\}$ and $\nabla u$ above is the gradient with respect to the independent variable of $u$ from $V^{-1}\Omega$, not with respect to $x\in\Omega$.

By orthogonal and translation invariance, it holds for almost all $x\in\Omega$ that 
\[
f(\nabla y(x))= g(Z(x))= |Z(x)-\Sigma|^2\,|Z(x)+\Sigma|^2 = f_C(Z(x)) + f_L(Z(x))
\]
where 
\[
f_C(Z):=|Z|^4
+ 2\sum_{i=1}^n \big(|\Sigma|^2-2\sigma_i^2\big)\,Z_{ii}^2
+ 2\sum_{i<j}
\begin{bmatrix}Z_{ij}&Z_{ji}\end{bmatrix}
M_{ij}
\begin{bmatrix}Z_{ij}\\ Z_{ji}\end{bmatrix}
+ |\Sigma|^4,
\qquad
f_L(Z):=-\,8\,s_2(Z)
\]
according to the algebraic decomposition \eqref{eq:gfull}.
By the spectral condition of Theorem~\ref{polyconvex}, $f_C$ is convex in $Z$. Moreover, since $|Z|^4$ is strictly convex,
$f_C$ is strictly convex. The term $f_L$ is affine in the $2\times2$ minors of $Z$.

As $|\det V|=1$ we have 
\begin{align*}
I(y)=\int_\Omega f(\nabla y(x))\, dx=\int_\Omega g(\nabla u(V^{-1}x))\, dx =\int_{V^{-1}\Omega}g(\nabla u(\tilde x))\, d\tilde x  :=J(u)
\end{align*}

Let
\[
J_C(u):=\int_{V^{-1}\Omega} f_C(\nabla u(x))\,dx, \qquad J_L(u):=\int_{V^{-1}\Omega} f_L(\nabla (x))\,dx
\]
so that
\[
J(u)= J_C(u)+J_L(u).
\]

We reformulate \eqref{minI} as a minimization problem for $J$:
\begin{align}\label{minJ}
	  \min_{u \in W^{1,4}(V^{-1}\Omega; \R^n)} J(u) \text{ s.t. } u( x) = U^\top y_0(Vx)-U^\top BVx\text{ for } x\in  \partial V^{-1}\Omega.   
	\end{align} 
 The problem \eqref{minJ} admits a unique solution. Indeed, 
by the classical characterization of null Lagrangeans (see, e.g., \cite[Thm.~5.21, Cor.~5.22]{Dacorogna2008}), we get  $J_L(u)=J_L(u_0)$  for all admissible $u$.  Here,
$u_0(x) = U^\top y_0(Vx)-U^\top BVx$ for almost all  $x\in V^{-1}\Omega$. This means that $J_L$ is constant in our minimization problem.  Moreover, $J$ is coercive on 
$W^{1,4}(V^{-1}\Omega; \R^n)$ and the admissible set of mappings on which we minimize is convex.  The existence of minimizers then follows by the direct method and the strict convexity of $J_C$ implies the uniqueness. Let us denote the minimizer by $\tilde u$. Consequently, 
$\tilde y(x)=U(\tilde u(V^{-1}x)+U^\top Bx)$ for $x\in\Omega$ is the unique minimizer of $I$  in \eqref{minI}.

\end{proof}

\section*{Conflict Of Interest Statement}
The authors have no conflict of interest.

\section*{Data Availability Statement}

This manuscript has no associated data.

\section*{Acknowledgements}
This work was co-funded by the European Union/M\v{S}MT \v{C}R  under the ROBOPROX project (reg.~no.~CZ.02.01.01/00/22 008/0004590). 
Moreover, the second author was also supported by the GA\v{C}R project 24-10366S and he  is indebted to  Prof.~Jonathan Bevan for valuable discussions.  We sincerely thank the anonymous reviewer for the insightful comments and for bringing to our attention the flaw in the original statement and proof of Theorem 1.

\end{document}